\documentclass[11pt,letterpaper]{amsart}
\usepackage{graphicx,mathtools,enumerate}
\usepackage{amssymb}
\usepackage{epstopdf}
\usepackage{varwidth}
\usepackage{cancel}
\usepackage[colorlinks=true, pdfstartview=FitV, linkcolor=blue, citecolor=blue, urlcolor=blue]{hyperref}
\usepackage{tikz-cd}
\usepackage{caption}
\usepackage{subcaption}
\usepackage{mathrsfs}
\usepackage{comment}
\usepackage{framed}
\usepackage{color}
\definecolor{shadecolor}{gray}{0.875}

\usepackage[margin=2.5cm]{geometry}

\usepackage{color}

\newcommand{\Z}{{\mathbb Z}}
\newcommand{\RR}{{\mathbb R}}

\def\Ass{\operatorname{Ass}}

\def\depth{\operatorname{depth}}

\def\NP{\operatorname{NP}}

\def\np{\operatorname{np}}

\newcommand{\factor}[2]{\left. \raise 1pt\hbox{\ensuremath{#1}} \right/
        \hskip -2pt\raise -3pt\hbox{\ensuremath{#2}}}

\theoremstyle{plain} 
\newtheorem{thm}{Theorem}[section]

\newtheorem{introthm}{Theorem}

\newtheorem*{introthm*}{Theorem}
\newtheorem{question}{Question}
\newtheorem{cor}[thm]{Corollary}
\newtheorem{lem}[thm]{Lemma}
\newtheorem{prop}[thm]{Proposition}

\theoremstyle{definition}
\newtheorem{defn}[thm]{Definition}

\newtheorem{ex}[thm]{Example}

\newtheorem{rem}[thm]{Remark}

\numberwithin{equation}{section}  

\makeatletter  
\@namedef{subjclassname@2020}{%
  \textup{2020} Mathematics Subject Classification}
\makeatother

\title{Newton Polytopes and Analytic Spread}
\author{Benjamin Drabkin}
\author {Benjamin Oltsik}

\address{Singapore University of Technology and Design,
8 Somapah Rd, Building 1, Singapore 487372}
\email{benjamin\_drabkin@sutd.edu.sg}

\address{Department of Mathematics, University of Connecticut, Storrs, CT 06269}
\email{benjamin.oltsik@uconn.edu}


\begin{document}
\begin{abstract} 
Using the Newton polytope and polyhedron, we study analytic spread and ideal reductions of monomial ideals.  We determine a bound for analytic spread based on halfspaces and hyperplanes of the Newton polytope, and we classify basic monomial ideals.   We then apply this method to calculate the analytic spread for a few families of monomial ideals.
\end{abstract}

\maketitle

\section{Introduction}
Let $K$ be an infinite field, and let $I\subseteq R=K[x_1,\dots,x_n]$.  The analytic spread, $\ell(I)=\dim R[It]\otimes_R K$, is an invariant which captures information about the generators of the powers of $I$. 
 Notably, $\ell(I)$ gives the minimal number of generators in a reduction of $I$.    

In \cite{BA}, Bivi\`{a}-Ausina proves that the analytic spread of a monomial ideal can be determined from its Newton polyhedron, $\NP(I)$.  In particular, the analytic spread is one higher than the maximum dimension of a compact face of $\NP(I)$.  This is a very powerful tool for computing analytic spread, as it allows one to use convex geometric methods to solve what would otherwise be algebraic questions.  When given a sketch of NP(I), one can immediately identify the analytic spread.  However, when NP(I) is implicitly given by halfspaces and hyperplanes, it is often complicated to determine which faces of the Newton polyhedron are compact.

In this paper, we show that the analytic spread of any monomial ideal can be characterized in terms of its Newton polytope, $\np(I)$, and consider applications of this characterization.  In Section 2, we give an overview of necessary background information.  In Section 3, we introduce and prove our characterization of the analytic spread of monomial ideals.
\begin{introthm}[Theorem \ref{thm:spreadCalc}]\label{introThm1}
    Let $I$ be a monomial ideal, and let $\np(I)$ have hyperplanes defined by $$\begin{array}{c}\mathbf{w}_1\cdot\mathbf{u}=b_1\\
    \mathbf{w}_2\cdot\mathbf{u}=b_2\\
    \vdots\\
    \mathbf{w}_s\cdot\mathbf{u}=b_s\end{array}$$
    and halfspaces defined by
    $$
    \begin{array}{c}
    \mathbf{h}_1\cdot\mathbf{u}\leq c_1\\
    \mathbf{h}_2\cdot\mathbf{u}\leq c_2\\
    \vdots\\
    \mathbf{h}_t\cdot\mathbf{u}\leq c_t
    \end{array}.
    $$

    Suppose that there exist $\alpha_1,\dots,\alpha_s\in\mathbb{R}$ and $\beta_1,\dots,\beta_t\in\mathbb{R}_{\geq 0}$ such that $$ \mathbf{W}:=\sum_{i=1}^s\alpha_i\mathbf{w}_i+\sum_{i=1}^t\beta_j\mathbf{h}_i$$ has all negative entries.  Then $\ell(I)\leq n+1-(s+k)$ where $k$ is the minimal number of non-zero $\beta_j$ required to achieve $\mathbf{W}$ with all negative entries.
\end{introthm}
In many cases, this provides an easier method of computing analytic spread than determining the compact faces of the Newton polyhedron.

In Section 4, we use this result to give a characterization of basic monomial ideals.
\begin{introthm}[Theorem \ref{thm:expectedCodim}]
Let $I$ be a monomial ideal in $K[x_1, \ldots, x_n]$, and assume $\np(I)$ lies on the $s$ hyperplanes, $\{\mathbf{w}_i \cdot \mathbf{u} = b_i : i = 1, \ldots, s\}$, where $\mathbf{w}_i = (a_{i, 1}, \ldots , a_{i, n})$, presented so that the $s \times n$ matrix of $\mathbf w_i$'s is in reduced row echelon form.  Then $I$ is basic if and only if all of the following conditions hold:
\begin{enumerate}
    \item $\mu(I) \le n$,
    \item $s = n - \mu(I) + 1$,
    \item for every $j \in \{1, \ldots, n\}$, there exists $i_j$ such that $a_{i_j, j} > 0$.
\end{enumerate}
\end{introthm}

In Section 5, we demonstrate applications of Theorem \ref{introThm1}.   We show that the inequality of \ref{introThm1} is an equality for all Newton polyhedra in three dimensions, and we compute the analytic spread of two families of monomial ideals. In particular, for monomial ideals with with $r$ disjointly generated associated primes, we show that $\ell(I)=n-r+1$.  For monomial ideals that are intersections of two prime powers, we show that $\ell(I)=n-1$.

\section{Background}
In this paper, all rings are commutative and unital.


\begin{defn}\label{def:red} Let $R$ be a ring, and let $J \subseteq I$ be ideals.  We call $J$ a \textit{reduction} of $I$ if, for some $n \in \Z_{>0}$, $JI^n = I^{n+1}$.
\end{defn}

\begin{ex} Let $R = K[x, y]$, $J = (x^m, y^m)$, and $I = (x, y)^m$ for any $m$.  It is clear that $J \subseteq I$ is a reduction since $JI = I^2$.
    
\end{ex}

A reduction of an ideal shares many traits with the original.  For example, an ideal and its reduction have the same radical, minimal primes, and height \cite{HS}.  Importantly, an ideal $J \subseteq I$ is a reduction if and only if $\overline{J} = \overline{I}$.

\begin{defn}\label{def:minred} Let $J$ be a reduction of $I$, then $J$ is a \textit{minimal reduction} of $I$ if it contains no proper reductions of $I$.  If $I$ has no reductions other than itself, then $I$ is called \textit{basic}.
    
\end{defn}


A main goal of this paper is to classify precisely when certain families of monomial ideals are basic.  To do this, we need the following:

\begin{defn}\label{def:reesAnSp} Let $I \subseteq R$ be an ideal.  The \textit{Rees algebra of $I$} is
\[
    \mathcal R(I) = \bigoplus_{n=1}^\infty I^n t^n.
\]
If $R$ is local (or graded) with maximal ideal (or unique maximal homogeneous ideal) $\mathfrak m$, the analytic spread is
\[
    \ell(I) = \dim\left(\mathcal R(I) \otimes_R \factor{R}{\mathfrak m} \right) = \dim\left(\bigoplus_{n=1}^\infty \frac{I^n}{\mathfrak m I^n} t^n\right),
\]
where $\dim(M)$ is the Krull dimension of $M$.
    
\end{defn}

An important fact that we will use later in the paper is that $\ell(I) \le \dim R$.

While not immediately apparent, there is a critical relation between analytic spread and reductions of ideals:

\begin{thm}[Corollary 8.3.6, Proposition 8.3.7 \cite{HS}] \label{thm:minred} Let $R$ be a local (or graded) ring, and let $I$ be a (homogeneous) ideal.  Then, if the residue field of $R$ is infinite and $J \subseteq I$ is a reduction, then $J$ is a minimal reduction of $I$ if and only if $\mu(J) = \ell(I)$.
\end{thm}

As an immediate consequence,

\begin{cor}\label{cor:basicCriterion} Let $R$ be a local (or graded) ring with infinite residue field, and let $I$ be a (homogeneous) ideal.  Then $I$ is basic if and only if $\mu(I) = \ell(I)$.

\end{cor}

We now introduce a few tools to help work with monomial ideals.

\begin{defn}\label{def:newtonStuff} Let $R = K[x_1, \ldots, x_n]$, and let $I$ be a monomial ideal.  The \textit{Newton polytope} of $I$ is
\[
    \np(I) := \operatorname{conv}\{\mathbf b \in \Z_{\ge 0}^n : {\mathbf x}^\mathbf b \in G(I)\},
\]
where $G(I)$ is the minimal generating set of monomials for $I$, $\mathbf b = (b_1, \ldots, b_n)$, and $\mathbf x^{\mathbf b} = x_1^{b_1}\cdots x_n^{b_n}$.  The vector $\mathbf b$ is called the \textit{exponent vector}.

Define the \textit{Newton polyhedron} of $I$ as
\[
    \NP(I) := \operatorname{conv}\{\mathbf b \in \Z_{\ge 0}^n : {\mathbf x}^\mathbf b \in I\}.
\]

Equivalently,
\[
    \NP(I) = \np(I) + \RR_{\ge 0}^n,
\]
where the above addition is the Minkowski sum.


\end{defn}
An important fact about $\NP(I)$ is that a monomial $\mathbf{x}^\mathbf{b} \in \overline I$ if and only if $\mathbf b \in \NP(I)$.  Thus, a monomial reduction of an ideal will have the same Newton polyhedron, but not necessarily the same Newton polytope.

As with any convex polyhedron, the Newton polytope and Newton polyhedron can be expressed as an intersection of a set of hyperplanes and halfspaces.  In particular, a minimal set of the hyperplanes on which the polytope lies is unique up to linear combinations.  Also, halfspaces can be chosen so that the intersection between the polytope and the supporting hyperplane of the halfspace is a facet of the polytope.  We call such a set of halfspaces, \textit{facet-defining}.

A remarkable result from Bivià-Ausina \cite{BA} relates analytic spread with the Newton polyhedron: 

\begin{thm}[Bivià-Ausina \cite{BA}]\label{thm:cpctFace} Let $I$ be a monomial ideal.  Then $\ell(I)$ is equal to one more than the largest dimension of a compact face of $\NP(I)$.
\end{thm}


\begin{ex}\label{ex:starter} Let $R = K[x, y, z]$, and let $I_1 = (xy, xz, yz)$ and $I_2 = (xy, y^4z^4, z^4x^4)$.  Using the notation that $(u, v, w)$ is the exponent vector corresponding with $x^uy^vz^w$, the bounding halfspaces for the Newton polyhedra are as follows:

\[
    \NP(I_1) = \begin{cases} u + v \ge 1\\
    u + w \ge 1\\
    v + w \ge 1\\
    u + v + w \ge 2\\
    u, v, w \ge 0,
    \end{cases}
\quad \text{and} \quad
    \NP(I_2) = \begin{cases}
        u + 3v \ge 4\\
        3u + v \ge 4\\
        4u + w \ge 4\\
        4v + w \ge 4\\
        u, v, w \ge 0
    \end{cases}
\]

Figures \ref{fig:easyExNP} and \ref{fig:analySpread2} are depictions of $\NP(I_1)$ and $\NP(I_2)$ respectively, as well as the corresponding Newton polytopes shaded.  Observe $\NP(I_1)$ has a compact triangle in on its boundary, so $\ell(I_1) = 3$.  This compact triangle is also $\np(I_1)$.  However, in Figure 2, $\NP(I_2)$ has no such two-dimensional compact boundary.  Instead, the largest dimension of a compact face of $\NP(I_2)$ is one, as highlighted in blue, so $\ell(I_2) = 2$.  Intuitively, this is because $\np(I_2)$ ``leans forward'' away from the origin, so when adding $\RR_{\ge 0}^3$, $\NP(I_2)$ takes up space behind $\np(I_2)$.
    
\end{ex}

\begin{minipage}{.5\textwidth}
  \centering
  \includegraphics[width=.6\linewidth]{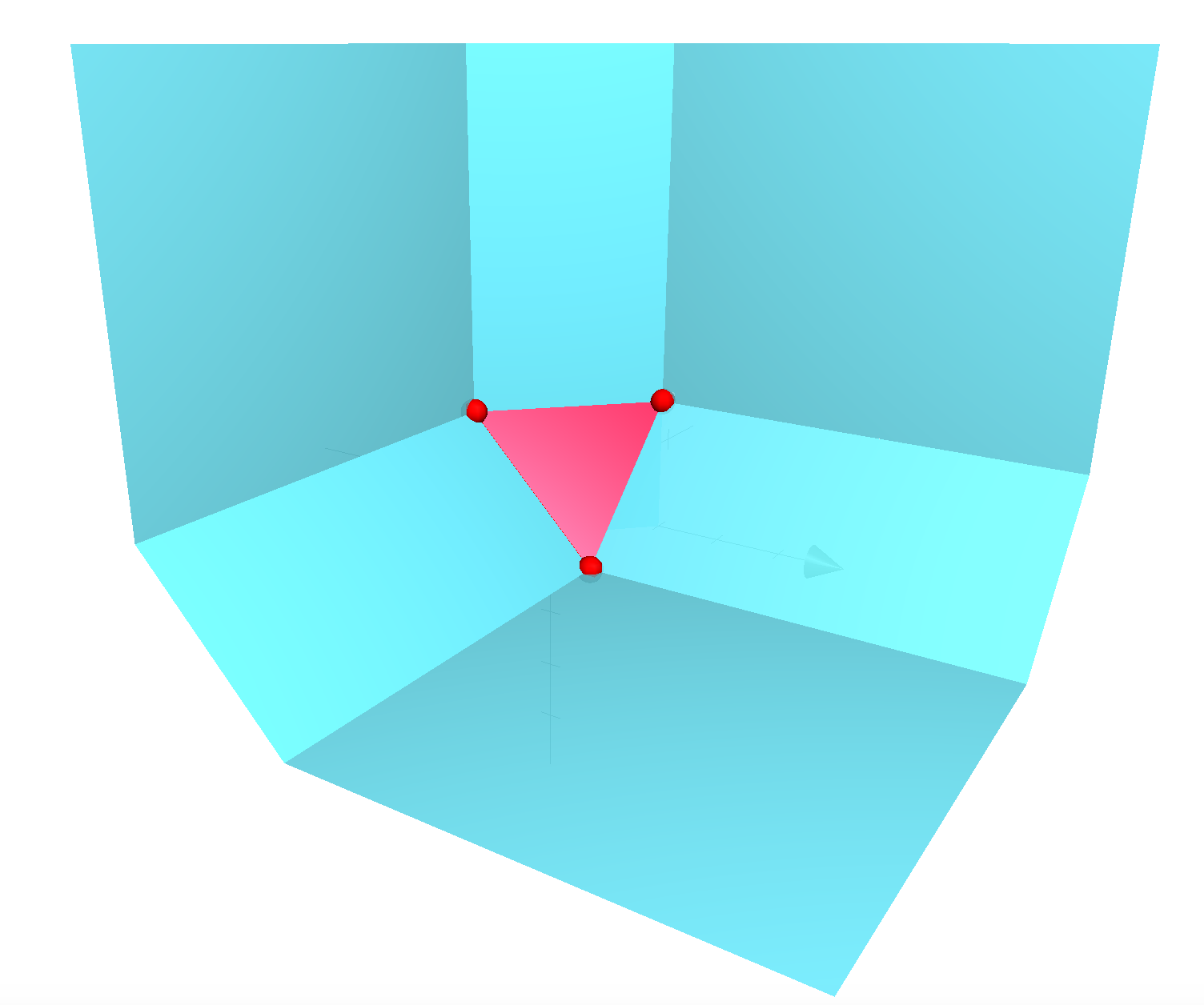}
  \captionof{figure}{${\rm NP}(xy, yz, xz)$}
  \label{fig:easyExNP}
\end{minipage}%
\begin{minipage}{.5\textwidth}
  \centering
  \includegraphics[width=.6\linewidth]{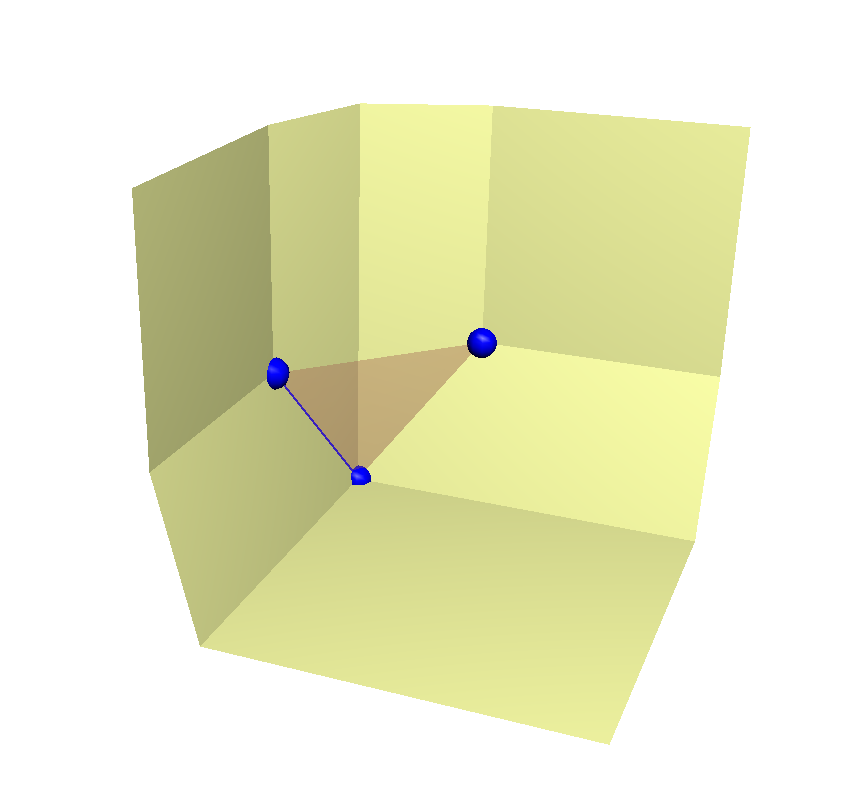}
  \captionof{figure}{$\NP(xy, y^4z^4, x^4z^4)$}
  \label{fig:analySpread2}
\end{minipage}

\begin{rem} In this paper, we will use $\mathbf u = (u_1, \ldots, u_n)$ to denote variables for the convex geometric objects, such as the exponent vectors for a ring in $n$ indeterminates.

\end{rem}

\section{Bounding the analytic spread with the Newton polytope}

In this section, we will detail a method using the equations of halfspaces and hyperplanes of the Newton polytope to bound the analytic spread.  We will prove this in more generality, as the result only requires that a polyhedron be the Minkowski sum of a polytope and the first orthant.

\begin{lem}\label{lem:cpctFaceSum} Let $E$ be a polytope, and let $N$ be the Minkowski sum, $N = E + \RR_{\ge 0}^n$.  Then, every compact face of $N$ is also a compact face of $E$.  
    
\end{lem}
\begin{proof} The facets of a Minkowski sum can be written as the Minkowski sum of faces.  In our case, the faces of $N$ are a Minkowski sum of faces of $E$ and $\RR_{\ge 0}^n$.  But the only compact face of $\RR_{\ge 0}^n$ is the origin, so every compact face of $N$ has a corresponding compact face in $E$.
\end{proof}

By definition, given a polyhedron $P$, $F$ is a face of $P$ if and only if there exists a linear functional $w$ and a real number $b$ such that $F = \{w(x) = b : x \in P\}$, and $w(y) \le b$ for all $y \in P$.

\begin{lem}\label{lem:strictDec} Let $E$ and $N$ be as in Lemma \ref{lem:cpctFaceSum}.  Let $F = \{w(x) = b\}$ be a face of $E$.  Then $F$ is a face of $N$ if and only if $w$ is strictly decreasing.  That is, for any coordinate vector $e_i$, $w(e_i) < 0$.  Alternatively, this is equivalent to every coefficient of $w$ being negative.
    
\end{lem}

\begin{proof}

Assume $F$ is a face of $N$.  Then $F = \{x \in N : w(x) = b\}$ and $w(y) \le b$ for all $y \in N$.  In other words, $w(x) = b$ if and only if $x \in F$.  Now, let $y = x + e_i$, where $x \in F$.  Then $w(y) < b$, but also $w(y) = w(x) + w(e_i) = b + w(e_i)$.  Thus, $w(e_i) < 0$.

Now suppose $w$ is strictly decreasing.  We aim to show that, for any $y \in N$, it follows $w(y) \le b$, with equality holding if and only if $y \in F$.  Let $y = x + v$, where $x \in E$, and $v \in \RR^n_{\ge 0}$.  By definition of strictly decreasing, $w(y) = w(x + v) \le w(x)$, with equality only when $v = 0$.  If $w(y) = b$, then we have $w(y) = b \le w(x) \le b$. So, $w(y) = w(x) = b$, attaining equality, and thus $y \in E$. That is, $F = \{ w(x) = b : x \in N\}$.  The condition that $w(y) \le b$ is evident.
\end{proof}

To get our final bound, we define one more object:

\begin{defn}
    Let $P\subseteq\mathbb{R}^n$ be a polyhedron and $F$ be a face of $P$.  Then the \textbf{normal cone} of $P$ at $F$ is $$N_P(F)=\{\mathbf{x}\in\mathbb{R}^n|\mathbf{x}\cdot(\mathbf{y}-\mathbf{z})\leq0\mbox{ for all }\mathbf{y}\in P,\mathbf{z}\in C\}.$$
    The \textbf{normal fan} $\mathcal{N}(P)$ is the fan consisting of the normal cones of $P$ at each of its faces.
\end{defn}
The normal fan $\mathcal{N}(P)$ is a complete fan (that is, the union of its cones is $\mathbb{R}^n$), its one-dimensional cones are the $\mathbf{h}\in\mathbb{R}^n$ where $\mathbf{h}\cdot\mathbf{x}\leq c$ is a facet-defining halfspace fo $P$ for some $c$.  The lineality space of $N(P)$ (the largest linear space contained in every one of its cones) is the span of the $\mathbf{w}\in\mathbb{R}^n$ such that $P$ is contained in the hyperplane $\mathbf{w}\cdot\mathbf{x}=c$ for some $c$.  . 

We now provide a way to extract analytic spread from the hyperplanes and halfspaces of the Newton polytope.

\begin{thm}\label{thm:spreadCalc}
    Let $E$ be a polytope, and let $N = E + \RR^n_{\ge 0}$, where the addition is the Minkowski sum.  Let $E$ have a minimal set of hyperplanes given by $$\begin{array}{c}\mathbf{w}_1\cdot\mathbf{u}=b_1\\
    \mathbf{w}_2\cdot\mathbf{u}=b_2\\
    \vdots\\
    \mathbf{w}_s\cdot\mathbf{u}=b_s\end{array}$$
    and (facet-defining) halfspaces given by
    $$
    \begin{array}{c}
    \mathbf{h}_1\cdot\mathbf{u}\leq c_1\\
    \mathbf{h}_2\cdot\mathbf{u}\leq c_2\\
    \vdots\\
    \mathbf{h}_t\cdot\mathbf{u}\leq c_t
    \end{array}.
    $$

    Suppose that there exist $\alpha_1,\dots,\alpha_s\in\mathbb{R}$ and $\beta_1,\dots,\beta_t\in\mathbb{R}_{\geq 0}$ such that $$ \mathbf{W}:=\sum_{i=1}^s\alpha_i\mathbf{w}_i+\sum_{i=1}^t\beta_j\mathbf{h}_j$$ has all negative coefficients.  Let $k$ denote the minimal number of non-zero $\beta_j$ terms needed to achieve all negative coefficients.  Then the maximal dimension of a compact face of $N$ is bounded above by $n - (s+k)$.

    In the case that $E = \np(I)$ for some monomial ideal $I$, $\ell(I) \le n + 1 - (s+k)$.

\end{thm}
\begin{proof}
Since $\mathcal{N}(E)$ is a complete fan, at least one of its cones intersects the inerior of the negative orthant, denoted $\mathfrak{O}$.  Let $C=\mbox{cone}(\mathbf{h}_{p_1}+\dots+\mathbf{h}_{p_m})+\mbox{span}(\mathbf{w}_{1}+\dots+\mathbf{w}_s)$ be minimum-dimensional cone of $N(E)$ which intersects $\mathfrak{O}$.  Then there is some point $\mathbf{o}\in\mathfrak{O}$ such that, for $\alpha_1,\dots,\alpha_s\in\mathbb{R}$ and positive $\beta_1,\dots,\beta_m$ we have $\mathbf{o}=\sum_{i=1}^s\alpha_i\mathbf{w_i}+\sum_{i=1}^m\beta_i\mathbf{h}_{p_i}$.  We note that, by assumption $m\geq k$.  Let $F$ be the face of $E$ such that $C=N_E(F)$, and let $B=\alpha_1b_1+\dots+\alpha_tb_t+\beta_1c_{p_1}+\dots+\beta_mc_{p_m}$.  Then, $F$ is the face cut out by $\mathbf{o}\cdot\mathbf{x}=B$, and is the face of $E$ lying on the hyperplanes defined by $\mathbf{h}_{p_i}\cdot\mathbf{x}=c_{p_i}$.  In particular, $F$ has codimension $(s+m)$.  Since all entries of $F$ are negative, we see that $\mathbf{o}$ defines a strictly decreasing linear functional, and thus $F$ is a face of $N$ as well.  Since $F$ is compact, we conclude that the maximum dimension of a compact face of $N$ is $n-(s+m)\leq n-(s+k)$.

\end{proof}

\begin{rem}\label{rem:nonEqual} The reason we have an inequality is that the argument above is contingent on observing the face of the Newton polytope induced by the intersection of $k$ of its faces.  This has, at least, codimension $k$, but there are examples in which the intersection of $k$ faces gives codimension greater than $k$.  Consider the ideal, $I = (x^{20}, x^8y^2, x^{14}yz^{10}, x^2y^{10}, x^5y^6z^{10}, y^{30}, xy^{20}z^{10})$.  Its Newton polytope and polyhedron are given in Figure \ref{fig:problemIdeal}.  
This Newton polytope is defined by the following inequalities:
$$\begin{cases}
    -u-6v\leq -20\\
    -4u-3v\leq -38\\
    -10u-v\leq -30\\
    -70u-20v+13w\leq 340\\
    -50u-90v+21w\leq -580\\
    -w\leq 0\\
    w\leq 10\\
    15u+10v+8w\leq 300\\
    190u+130v+111w\leq 3900
\end{cases}$$
We note that an all-negative vector can be constructed using the coefficient vector of one of the first five inequalities along with he coefficient vector of the sixth.  However, not all of the faces defined by the intersections with the supporting hyperplanes of these halfspaces have the same codimension.  The intersection of the Newton polytope with the hyperplanes defined by one of $ -u-6v= -20, -4u-3v= -38$, or $-10u-v= -30$ along with $-w=0$ have codimension two (and correspond to the red edges in Figure \ref{fig:problemIdeal}). The intersection of the Newton polytope with the hyperplanes defined by one of $ -70u-20v+13w= 340$ or $-50u-90v+21w= -580$ along with $-w=0$ have codimension three (and correspond to the marked points in Figure \ref{fig:problemIdeal}).  
    
\end{rem}

\begin{figure}
    \centering
    \includegraphics[width=0.5\linewidth]{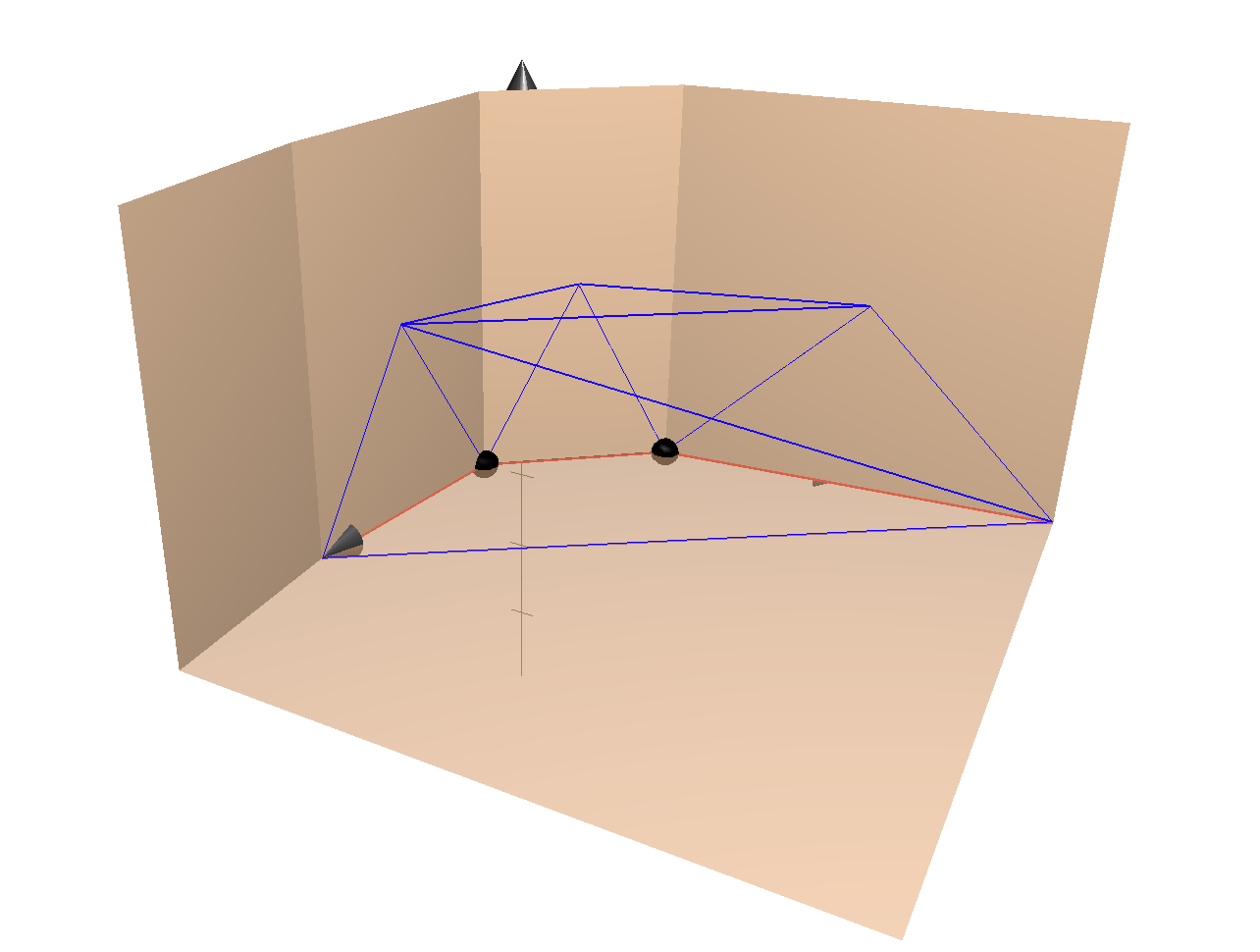}
    \caption{Newton polytope and polyhedron of $I$ from Remark \ref{rem:nonEqual}}
    \label{fig:problemIdeal}
\end{figure}

\begin{ex}
    Returning to Example \ref{ex:starter}, we note that the newton polytope for $I_1=(xy,xz,yz)$ is a 2-dimensional polytope contained in the hyperplane defined by $u+v+w=2$, and thus by Theorem \ref{thm:spreadCalc} it follows that $\ell(I_1)=3$.  For $I_2=(xy,y^4z^4,x^4z^4)$, we note that the newton polytope has halfspaces \[\begin{cases}
        u  - v -w \le 0\\
        -u  + v -w\le 0\\
        -u  -v +w \le 0
   
    \end{cases}\]
    and hyperplane \[-2u-2v+w=-4.\] Since the sum $2(-2u-2v+w) + 3(u-v-w)=-u-7v-w$ has all negative coefficients, it follows that $\ell(I_2)\le2$. However, since $I_2$ is not principal and not zero, $\ell(I_2) = 2$ (see proof of \ref{thm:trivBasic}).


\end{ex}

\begin{rem} For an arbitrary graded family of monomial ideals, one can define an analogue of Rees algebras and analytic spread. There are also polyhedra that associated with general graded families of monomial ideals called \textit{Newton-Okounkov bodies}.  Remarkably, Bivià-Ausina's theorem generalizes, in that analytic spread of a general graded family is with one more than maximal dimension of a compact face of the Newton-Okounkov body.  In many cases, Newton-Okounkov bodies can be written as a Minkowski sum of the form in Theorem \ref{thm:spreadCalc}.  Thus, one can use Theorem \ref{thm:spreadCalc} to calculate analytic spread.  For more details on Newton-Okounkov bodies and analytic spread of graded families, the authors recommend \cite{HN}.
    
\end{rem}

\section{Basic monomial ideals}

In this section, we aim to categorize precisely when a monomial ideal is basic.  First, let us discuss the cases of ideals with a small number of minimal generators.  

\begin{thm}\label{thm:trivBasic} If $I$ is an ideal in a polynomial ring with infinite residue field such that $\mu(I) \le 2$, then $I$ is basic.
    
\end{thm}

\begin{proof} By Corollary \ref{cor:basicCriterion}, it is enough to show $\ell(I) = \mu(I)$.

First, note that if $\ell(I) = 0$ then $I$ is nilpotent, so in a polynomial ring, the only ideal with analytic spread zero is the zero ideal.  Furthermore, from \cite{Huckaba}, the only ideals of a polynomial ring with analytic spread one are principal.  Thus, any  ideal with $\ell(I) = 1$ is also basic.

Finally, if $\mu(I) = 2$, recall $\ell(I) \le \mu(I) = 2$.  Since $I$ is non-zero and non-principal, $\ell(I) \ge 2$, so $\ell(I) = \mu(I)$.
    
\end{proof}

This allows for the polynomial ring with two variables to be a trivial case:

\begin{thm} Let $I$ be a monomial ideal in $K[x, y]$, $K$ infinite.  Then $I$ is basic if and only if $\mu(I) \le 2$.
    
\end{thm}
\begin{proof} Employ Theorem \ref{thm:trivBasic}, and the fact that $\ell(I) \le \dim R = 2$.
    
\end{proof}

In a polynomial ring three or more variables, we get our first instances of non-trivial basic ideals.  We will prove a criterion for precisely when a monomial ideal is basic shortly.


\begin{lem}\label{lem:numHP} Let $I$ be a monomial ideal in a polynomial ring in $n$ variables over infinite residue field.  Let $m = \mu(I)$, and let $\{\mathbf{w}_i \cdot \mathbf u = b_i : i = 1, \ldots, s\}$ be a minimal set of equations of the hyperplanes on which $\np(I)$ lies.  If $I$ is basic, then $s = n - m + 1$.
    
\end{lem}

\begin{proof}  If the dimension of $\np(I)$ is less than $m - 1$ (i.e. $\np(I)$ lies on more than $n - m + 1$ hyperplanes), then no compact face of $\NP(I)$ can have a dimension of $m - 1$.  Thus $\ell(I) \ne \mu(I)$, so $I$ is not basic.
    
\end{proof}

From this it follows:

\begin{lem}\label{lem:npMax} Let $I$ be a monomial ideal in a polynomial ring in $n$ variables over an infinite field.  with $m = \mu(I)$ such that $\np(I)$ lies on the intersection of $n - m + 1$ hyperplanes.  Then $I$ is basic if and only if the maximum dimensional compact face of $\NP(I)$ is $\np(I)$.
\end{lem}

Now, we have a main result:

\begin{thm}\label{thm:expectedCodim} Let $I$ be a monomial ideal in $K[x_1, \ldots, x_n]$, and assume $\np(I)$ lies on the $s$ hyperplanes, $\{\mathbf{w}_i \cdot \mathbf{u} = b_i : i = 1, \ldots, s\}$, where $\mathbf{w}_i = (a_{i, 1}, \ldots , a_{i, n})$, presented so that the $s \times n$ matrix of $\mathbf w_i$'s is in reduced row echelon form.  Then $I$ is basic if and only if all of the following conditions hold:
\begin{enumerate}
    \item $\mu(I) \le n$,
    \item $s = n - \mu(I) + 1$,
    \item for every $j \in \{1, \ldots, n\}$, there exists $i_j$ such that $a_{i_j, j} > 0$.
\end{enumerate}
\end{thm}

\begin{proof} Recall that $\ell(I) \le \dim R$, so if $\mu(I) > \dim R$, $I$ is not basic.

Let $A$ be the $s \times n$ matrix with rows $\mathbf{w}_i$.  We may assume $A$ is in reduced row echelon form. Note $A$ cannot have a row of zeroes, for if so, one of the hyperplanes would be redundant. Consider a linear combination of the rows: $\mathbf W = \sum_{i=1}^s \alpha_i \mathbf w_i$.

If $I$ is basic, by Theorem \ref{thm:spreadCalc} and Lemma \ref{lem:numHP}, there exist $\alpha_i \in \RR$ so that $\mathbf W$ has all negative entries.  Since the first $s$ diagonal entries of $A$ are 1, this means that all $\alpha_i$ are negative.  If, say, $a_{j, s+1} \le 0$ for all $j \in \{1, \ldots, n\}$, then, since $\alpha_i < 0$ for all $i$, the $(s+1)$-th entry of $\mathbf W$ is greater than 0, a contradiction.  

Similarly, if we assume $a_{j, s+1} \le 0$ for all $j$, then no $\alpha_i$ exist so that $\mathbf W$ could have all negative entries. But since $s = n - \mu(I) + 1$, by Theorem \ref{thm:spreadCalc}, this means $\ell(I) \le n + 1 - (s + k) = n + 1 - (n - \mu(I) + 1 + k) = \mu(I) - k$ for some $k > 0$.  That is, $\ell(I) < \mu(I)$, so $I$ is not basic.
    
\end{proof}

In the specific case of $\mu(I) = n$, we have

\begin{cor}\label{cor:eqDim} Let $I = (\mathbf{x}^{\mathbf{b}_1}, \ldots, \mathbf{x}^{\mathbf{b}_n})$ be a monomial ideal in $K[x_1, \ldots, x_n]$, $K$ infinite residue field.  Assume $\mathbf{b}_1, \ldots, \mathbf{b}_n$ lie on a unique hyperplane, given by $\mathbf w \cdot \mathbf u = b$.  Then $I$ is basic if and only if every entry of $\mathbf w$ is negative.
    
\end{cor}

\begin{proof} Assume that $I$ is basic.  Then $\np(I)$ is the unique compact facet of $\NP(I)$ of with maximal dimension.  In particular, since $\np(I)$ is itself a facet of $\NP(I)$, $\mathbf w$ has all negative coefficients by \ref{lem:strictDec}.

Conversely, if all entries of $\mathbf w$ are negative, then $\mathbf w$ is a strictly decreasing functional, so $\np(I)$ is a facet of $\NP(I)$ and thus $\ell(I) = n$.
\end{proof}

Thus, for a polynomial ring of three variables,

\begin{thm} Let $I$ be a monomial ideal in $K[x, y, z]$.  Then $I$ is basic if and only if $\mu(I) \le 2$ or $\mu(I) = 3$ and satisfies the conditions of Corollary \ref{cor:eqDim}.
\end{thm}

\section{On the Equality of Theorem \ref{thm:spreadCalc}}

As stated in Remark \ref{rem:nonEqual}, it is conjectured that Theorem \ref{thm:spreadCalc} is an equality, rather than an inequality.  This section will give several examples in which we show equality.

\subsection{Three-dimensional Case} In the case where $I$ is a monomial ideal in $K[x,y,z]$, Theorem \ref{thm:spreadCalc} allows us to calculate $\ell(I)$ exactly.

 Let $\mathfrak{O}$ denote the interior of the negative orthant of $\mathbb{R}^n$.  For a polyhedron $P$, let $\partial(P)$ denote the boundary of $P$ (the union of all proper faces of $P$).

\begin{lem}\label{mixedEntries}
    Let $P=\np(I)$ be a Newton polytope, and let $N=N(P)$ be its normal fan.  Let $H$ be a defining halfspace of some cone in $N$ and suppose that $H$ is defined by $\mathbf{h}\cdot\mathbf{x}\leq 0$.  Then $\mathbf{h}$ has both positive and negative entries.
\end{lem}
\begin{proof}
    The halfspaces of a cone in $N$ are given by inequalities of the form $(\mathbf{u}-\mathbf{v})\cdot\mathbf{x}\leq 0$, where $\mathbf{u}$ and $\mathbf{v}$ are vertices of $P$.  Since $P$ is a Newton polytope, no two vertices $\mathbf{u}$ and $\mathbf{v}$ can satisfy $(\mathbf{u}-\mathbf{v})_i\geq 0$ for all $i$.
\end{proof}

\begin{lem}\label{negOrth}
    Let $I$ be a non-principal ideal, $P=\np(I)$ be a newton polytope, and let $N=N(P)$ be its normal fan.  Suppose that $C$ is a maximal cone in $N$, satisfying $C\cap \mathfrak{O}\neq\emptyset$.  Then $\partial(C)\cap \mathfrak{O}\neq\emptyset$.
\end{lem}
\begin{proof}
    Let $H$ be a defining halfspace of $C$.  Then by Lemma \ref{mixedEntries}, we know that $H$ is defined by $\mathbf{h}\cdot\mathbf{x}\leq 0$ for some $\mathbf{h}$ which includes positive and negative entries.  Without loss of generality, we may assume that $\mathbf{h}$ has the form $$\mathbf{h}=(a_1,\dots,a_s,-b_1,\dots,-b_t,0,\dots,0),$$ where $a_1,\dots,a_s,b_1,\dots,b_t>0$.  Let $$\mathbf{y}=(-\frac{1}{sa_1},\dots,-\frac{1}{sa_s},-\frac{10}{b_1},\dots,-\frac{10}{b_t},-1,\dots,-1)\in \mathfrak{O}.$$  Then $\mathbf{h}\cdot\mathbf{y}=10t-1\geq 9$.  So $\mathbf{y}\in \mathfrak{O}\setminus C$.
    Let $\mathbf{z}\in C\cap \mathfrak{O}$. Then, since $C$ is full-dimensional, the line segment between $\mathbf{z}$ and $\mathbf{y}$ intersects $\partial(C)$.  Since $\mathfrak{O}$ is convex, this line segment is contained in $\mathfrak{O}$, so $\partial(C)\cap \mathfrak{O}\neq\emptyset$. 
\end{proof}
As a result of these lemmas, we can compute the analytic spread of monomial ideals in three-dimensional polynomial rings.

\begin{prop}
    Let $R$ be a three-dimensional polynomial ring and $I\subseteq R$ be a monomial ideal.  Then the inequality from Theorem \ref{thm:spreadCalc} is an equality.
\end{prop}
\begin{proof}

Let $P=\np(I)$.  Let $s$ and $k$ be defined as in Theorem \ref{thm:spreadCalc}.

If $I$ is a principal ideal, then $P$ is a point in $\mathbb{R}^3$ and is thus defined by the intersection of three hyperplanes, so $n+1-(s+k)=1$.  Since $\ell(I)=1$ for all principal ideals, the proposed equality holds.

Suppose $I$ is non-principal.  Then $\ell(I)\geq 2$.  If $s+k=1$, then $P$ has some halfspace (or hyperplane) of $P$ defined by $\mathbf{h}\cdot\mathbf{x}\leq c$ ($\mathbf{h}\cdot\mathbf{x}= c$, respectively) for some $c\in\mathbb{R}$ such that all entries of $\mathbf{h}$ are negative.  The face of $P$ defined by $P\cap\{\mathbf{x}|\mathbf{h}\cdot\mathbf{x}=c\}$ is a 2-dimensional face defined by a strictly decreasing linear function, and thus by Lemma \ref{lem:strictDec} is also a compact face of $\NP(I)$.

Suppose $s+k=2$.  Since $\mathcal{N}(P)$ is a complete fan, it contains a three-dimensional cone, $C$, which intersects $\mathfrak{O}$.  By Lemma \ref{negOrth}, we know that $\partial(C)$ intersects $\mathfrak{O}$, so there is a 2-dimensional face, $C'$ of $C$ such that $C'\cap\partial(C)\neq\emptyset$.  If $P$ is full-dimensional, then $C'=\mbox{cone}(\mathbf{h}_1,\mathbf{h}_2)$ for some $\mathbf{h}_1,\mathbf{h}_2$.  If $P$ is codimension $1$, then $C'=\mbox{cone}(\mathbf{h}_1)+\mbox{span}(\mathbf{h}_2)$ for some $\mathbf{h}_1,\mathbf{h}_2$.  If $P$ is codimension 2 then  for some $\mathbf{h}_1,\mathbf{h}_2$.  Since $C'$ intersects $\mathfrak{O}$, there exist $\alpha_1,\alpha_2$, such that $\alpha_1\mathbf{h}_1+\alpha_2\mathbf{h}_2\in\mathfrak{O}$.  If $P$ is full dimensional, then both $\alpha_i$ are positive. 
 If $P$ is codimension-1, then $\alpha_1$ is positive.  Thus there is a 1-dimensional face of $P$ defined by $P\cap\{\mathbf{x}|\mathbf{h}_1\cdot\mathbf{x}=c_1\}\cap\{\mathbf{x}|\mathbf{h}_2\cdot\mathbf{x}=c_2\}$, which is by Lemma \ref{lem:strictDec} also a compact face of $\NP(I)$.
\end{proof}

\subsection{Ideals with Disjointly Generated Primary Decomposition}

Call two monomial ideals $I$ and $J$ \textit{disjointly generated} if the set of variables in the minimal generation of $I$ and the set of variables in the minimal generation of $J$ are disjoint.  Note that if $I$ and $J$ are disjointly generated, then $I \cap J = IJ$.  The analytic spread of these ideals were studied in \cite{HV}, and analytic spread was precisely calculated to be $\ell(IJ) = \ell(I) + \ell(J) - 1$.  We will show that Theorem \ref{thm:spreadCalc} attains equality for ideals with disjointly generated primary intersection.



First, we establish how embedding monomial ideals into larger polynomial rings affects Newton polytopes and spread.

\begin{lem}\label{lem:embedding} Let $I$ be a monomial ideal in $K[x_1, \ldots, x_n]$, and let $\np(I)$ have hyperplanes given by $$\begin{array}{c}\mathbf{w}_1\cdot\mathbf{u}=b_1\\
    \mathbf{w}_2\cdot\mathbf{u}=b_2\\
    \vdots\\
    \mathbf{w}_s\cdot\mathbf{u}=b_s\end{array}$$
    and (facet-defining) halfspaces given by
    $$
    \begin{array}{c}
    \mathbf{h}_1\cdot\mathbf{u}\leq c_1\\
    \mathbf{h}_2\cdot\mathbf{u}\leq c_2\\
    \vdots\\
    \mathbf{h}_t\cdot\mathbf{u}\leq c_t
    \end{array}.
    $$

Let $S = K[x_1, \ldots, x_n, y_1, \ldots, y_m]$.  Let $\mathbf z$ be the exponent vector of coordinates for a monomial in $S$, and let $v_i$ denote the coordinate corresponding to the exponent of $y_i$.  Then $\np(IS)$ has hyperplanes given by
$$\begin{array}{c}\mathbf{w}_1'\cdot\mathbf{z}=b_1\\
    \mathbf{w}_2'\cdot\mathbf{z}=b_2\\
    \vdots\\
    \mathbf{w}_s'\cdot\mathbf{z}=b_s\\
    v_1 = 0\\
    \vdots\\
    v_m = 0\\
    \end{array}$$
    where $\mathbf{w}_i'$ is the vector with first $n$ entries equal to $\mathbf{w}_i$ and remaining entries equal to zero.  Furthermore, $\np(IS)$ has (facet-defining) halfspaces given by
    $$
    \begin{array}{c}
    \mathbf{h}_1'\cdot\mathbf{z}\leq c_1\\
    \mathbf{h}_2'\cdot\mathbf{z}\leq c_2\\
    \vdots\\
    \mathbf{h}_t'\cdot\mathbf{z}\leq c_t
    \end{array},
    $$
    where $\mathbf{h}_i'$ is the vector with first $n$ entries equal to $\mathbf{h}_i$, and remaining entries equal to zero.
    
\end{lem}
\begin{proof}  Every generator of $I$ when included in $K[x_1, \ldots, x_n, y_1, \ldots, y_m]$ has corresponding exponent vectors lying on $v_j = 0$ for $j = 1, ..., m$.  Thus it is evident that $\np(I)$ lies on the hyperplanes, $v_j = 0$.  All generators also clearly satisfy $\mathbf{w}_i \cdot \mathbf z = b_i$, so these are in the defining hyperplanes as well.  These are all possible hyperplanes, as their intersection gives a polytope of the same dimension as the original $\np(I)$.

Since $\mathbf{h}_i \cdot \mathbf u \le c_i$, it is evident that $\mathbf{h}_i' \cdot \mathbf z \le c_i$.  As the halfspaces are facet-defining for $\np(I)$, and $\np(IS)$ is an embedding of $\np(I)$ into a larger space, the number of facets should be retained.  So no extra halfspaces are necessary to define $\np(IS)$.
    
\end{proof}

As a result of this lemma, it follows that extending a monomial ideal to a ring with more variables does not change its analytic spread.  In particular, we conclude the following:

\begin{thm}\label{thm:embed} Let $I$ be a monomial ideal in $K[x_1, \ldots, x_n]$.  If $S = K[x_1, \ldots x_n, y_1, \ldots, y_m]$, then $\ell(IS) = \ell(I)$.
\end{thm}

Now, we turn our attention ideals that are intersections of disjointly generated ideals.  In order to apply Theorem \ref{thm:spreadCalc}, we must be able to determine the halfspaces and hyperplanes of the Newton polytope of the intersection of disjointly generated monomial ideals.

\begin{lem}\label{lem:disGen} Let $I$ and $J$ be two disjointly generated monomial ideals.  Then the hyperplanes (and facet-defining halfspaces) of $\np(I \cap J)$ is the union of the hyperplanes (and facet-defining halfspaces) of $\np(I)$ and $\np(J)$.  
    
\end{lem}
\begin{proof} Let $R = K[x_1, \ldots, x_n, y_1, \ldots, y_m]$. Assume $I$ is generated only by monomials with variables $x_i$, and $J$ is generated only by monomials of $y_j$.  Letting $\mathbf z$ be the vector of all exponent coordinates, let the hyperplanes of $\np(I)$ be denoted by $\{\mathbf w_i \cdot \mathbf z = b_i\}$, those of $\np(J)$ to be $\{\mathbf v_j \cdot \mathbf z = d_j\}$.  Let $I = (\mathbf{x}^{\alpha_1}, \cdots, \mathbf{x}^{\alpha_w})$, $J = (\mathbf{y}^{\beta_1}, \cdots, \mathbf{y}^{\beta_v})$.  Since $IJ = I\cap J$, $I \cap J = (\mathbf{x}^{\alpha_i}\mathbf{y}^{\beta_j} : 1 \le i \le w, 1 \le j \le v)$.  Note that all the coefficients of $\mathbf w_i$ corresponding to $y_j$ terms are zero, and all of the coefficients of $\mathbf v_j$ corresponding to $x_i$ terms are zero.  Thus, every generator of $I\cap J$ has an exponent vector satisfying the equations of both the hyperplanes of $\np(I)$ and $\np(J)$, so the hyperplanes of $\np(I \cap J)$ contain those of $\np(I)$ and $\np(J)$.

Now, consider an arbitrary hyperplane of $\np(I \cap J)$: $\mathbf p \cdot \mathbf z = e$.  We rewrite this as $\mathbf p_x \cdot \mathbf u_{x} + \mathbf p_y \cdot \mathbf u_{y} = e$, where $\mathbf p_x$ is the vector of the coefficients of $\mathbf p$ for the $x$ exponent vectors and zeroes for the $y$ exponent vectors and vice versa for $\mathbf p_y$, and $\mathbf u_{x}$ is the vector of $x$ exponents and zeroes for the $y$ exponents and vice versa for $\mathbf u_{y}$.  This hyperplane must contain every generator; in particular, it must contain every $\alpha_1\beta_j$ for all $j$.  This means that $ \mathbf p_y \cdot \beta_j = e - \mathbf p_x \cdot \alpha_1$.  That is, $\mathbf p_y \cdot \beta_j$ is constant for all $j$.  This means that $\mathbf p_y \cdot \mathbf z =  e - \mathbf u_x \cdot \alpha_1$ is a hyperplane of $\np(J)$.  We conclude that the hyperplanes of $\np(I \cap J)$ are precisely those of $\np(I)$ and $\np(J)$.

An analogous argument holds to show that the facet-defining halfspaces of $\np(I \cap J)$ are precisely those of $\np(I)$ and $\np(J)$.
    
\end{proof}

\begin{lem}\label{lem:primary} Let $\mathfrak q$ be a $\mathfrak p$-primary monomial ideal in $K[x_1, \ldots, x_n]$.  Then $\ell(\mathfrak q) = \mu(\mathfrak p)$.

\end{lem}
\begin{proof} First assume that $\mathfrak q$ is an $\mathfrak m$-primary monomial ideal.  Ideals and their reductions share the same minimal primes (Lemma 8.1.10, \cite{HS}), so if $\mathfrak q$ is an $\mathfrak m$-primary ideal, $\overline{\mathfrak q}$ is also $\mathfrak m$-primary.  By Theorem 5.4.6 of \cite{HS}, $\mathfrak m \in \operatorname{Ass}(R/\overline{I^m})$ for some $m$ if and only if $\ell(I) = n$.  Thus, $\ell(\mathfrak q) = n$.

Now assume $\mathfrak q$ is $\mathfrak p$-primary, $\mathfrak p \ne \mathfrak m$.  In $K[\mathfrak p]$, the ring with only the variables of $\mathfrak p$ as indeterminates, $\mathfrak qK[\mathfrak p]$ has analytic spread $\mu(\mathfrak p)$ by the above argument.  But by Theorem $\ref{thm:embed}$, this means $\ell(\mathfrak q) = \mu(\mathfrak p)$.
    
\end{proof}

We now describe the analytic spread of ideals with disjointly generated primary decompositions.

\begin{thm} \label{thm:disGenPrimary} Let $I = \bigcap\limits_{i=1}^r \mathfrak q_i$, where each $\mathfrak q_i$ is monomial and $\mathfrak p_i$-primary, each $\mathfrak p_i$ is disjointly generated from the others, and every variable is a generator of some $\mathfrak p_i$.  Then $\ell(I) = n - r + 1$.
\end{thm}

\begin{proof} Let us first consider $\mathfrak q_i$ in $K[\mathfrak p_i]$.  By Lemma \ref{lem:primary}, each $\mathfrak q_i$ has analytic spread $\mu(\mathfrak p_i)$.  Considering this in $K[\mathfrak p_i]$, applying  Theorem \ref{thm:spreadCalc}, $\ell(\mathfrak q_i) = \mu(\mathfrak p_i) = \mu(\mathfrak p_i) + 1 - s_i - k_i$, where $s_i$ and $k_i$ are as in $\ref{thm:spreadCalc}$.  Thus, $s_i + k_i = 1$, so $s_i = 1, k_i = 0$ or $k_i = 1, s_i = 0$.  If the former, there is exactly one hyperplane that $\np(\mathfrak q_i)$ lies on, and that hyperplane acts as a strictly decreasing linear functional.  If the latter, $\np(\mathfrak q_i)$ does not lie on a single hyperplane, and has associated to it one halfspace that acts as a strictly decreasing linear functional.

Thus, from Lemma \ref{lem:embedding} and Lemma \ref{lem:disGen}, to each $\mathfrak q_i$, there is either a hyperplane or a halfspace of $\np(\mathfrak q_i)$, with all negative coefficients.   Say there are $m$ with the hyperplane condition.  Take $\mathbf W$ to be the sum of these.  So $\ell(I) = n + 1 - (m - (m - r)) = n - r + 1$.
    
\end{proof}

\begin{ex} Let $R = K[x_1, \ldots, x_n]$, $I = \bigcap\limits_{i=1}^r \mathfrak q_i^{m_i}$, $m_i > 0$, where each $\mathfrak q_i = (x_{i1}^{e_{i1}}, \ldots, x_{ih_i}^{e_{ih_i}})$, monomial ideals generated by pure powers of variables.  Let $\mathfrak p_i = \sqrt{\mathfrak q_i} = (x_{i1}, \ldots, x_{ih_i})$, and assume that $\mathfrak p_i$ and $\mathfrak p_j$ are disjointly generated, and that every variable of the polynomial ring appears as a minimal generator of some $\mathfrak p_i$.

Now, by Proposition 8.1.5 of \cite{HS}, $\mathfrak q_i ^{[m_i]}$, the $m_i$-th Frobenius power, is a reduction of $\mathfrak q_i^{m_i}$.  Furthermore, since $\mathfrak q_i^{m_i}$ and $\mathfrak q_j^{m_j}$ have disjoint generating sets, $\mathfrak q_i^{m_i} \cap \mathfrak q_j^{m_j} = \mathfrak q_i^{m_i}\mathfrak q_j^{m_j}$.  Therefore, by Proposition 8.1.7 of \cite{HS}, $J = \bigcap\limits_{i = 1}^r \mathfrak q_i^{[m_i]}$ is a reduction of $I$, so it suffices to study $\ell(J)$.

The ideal $J$ is generated by all $\prod\limits_{i=1}^r h_i$ possible combinations of generators chosen from each generating set of the $\mathfrak q_i$, and each of the generators is raised to the appropriate power $m_i$.  So certainly, every generator of $J$ lies on the hyperplanes $\left\{\mathbf w_i \cdot \mathbf u := \frac{u_{i1}}{e_{i1}} + \frac{u_{i2}}{e_{i2}} + \cdots + \frac{u_{ih_i}}{e_{ih_i}} = m_i\right\}$, and thus $\{\mathbf w_i \cdot \mathbf u = m_i\}$ is a subset of the hyperplanes of $\np(J)$.  Letting $\mathbf W = \sum_{i = 1}^r -\mathbf w_i$, we have a strictly decreasing linear functional.  Thus, by Theorem \ref{thm:spreadCalc}, $\ell(I) = \ell(J) \le n + 1 - r$, and from \cite{HV}, we know this to be an equality.

\end{ex}

Interestingly enough, the analytic spread in the above example is independent of the powers of the generators, $e_{ij}$, the powers of the $\mathfrak {q_i}$, $m_i$, and the heights of the primes, $h_i$, and is only dependent on the number of primary ideals, $r$.  


\subsection{Intersections of powers of two monomial primes}

The obvious next step is to see what happens when the primary decomposition is not disjointly generated.  This proves much more difficult, as Lemma \ref{lem:disGen} does not apply, so it is not easy to find the Newton polytope in general.  However, in the following special case, we can apply Theorem \ref{thm:spreadCalc}.

\begin{thm}\label{thm:twoPrimes}
    Let $R = K[x_1, \ldots, x_s, y_1, \ldots, y_t, z_1, \ldots z_r]$. Let $\mathfrak{p}=(x_1,\dots,x_s,y_1,\dots,y_t)$ and $\mathfrak{q}=(y_1,\dots, y_t,z_1,\dots z_r)$, and let $I=\mathfrak{p}^a\cap\mathfrak{q}^b$ for some $a, b \ge 1$.  Then $\ell(I)\le n - 1$, where $n = r + s + t$ is the dimension of $R$.
\end{thm}
\begin{proof}
    We first note that $\np(I)$ has codimension at most 2.  Indeed, for $i=1,\dots,t$, $j=2,\dots,r$ and $k=1,\dots,s$, the monomials $y_k^a$, $x_1^az_j^b$, and $x_j^az_1^b$ are generators of $I$, and the exponent vectors of these generators are affinely independent.

    Now we consider the case where $a=b$. Let $\mathbf{h}_\mathfrak{p}$ be the vector with $1$ in each coordinate corresponding to $x_i$ or $y_j$ and 0 elsewhere.  Let $\mathbf{h}_\mathfrak{q}$ be the vector with $1$ in each coordinate corresponding to $z_i$ or $y_j$ and 0 elsewhere.  Then the exponent vector of every generator in $I$ satisfies $\mathbf{h}_{\mathfrak{p}} \cdot\mathbf{u}=a$ and $\mathbf{h}_{\mathfrak{q}} \cdot\mathbf{u}=a$.  Since the codimension of $\np(I)$ is at most 2, we see that these are exactly the hyperplanes of $\np(I)$.  Thus, by Theorem \ref{thm:spreadCalc}, $\ell(I)=n-1$. 

    Now suppose that $a>b$.  We claim that $\np(I)$ is contained in the hyperplane defined by $\mathbf{h}_\mathfrak{p}\cdot\mathbf{u}=a$.  The generators of $I$ are of the form $\mbox{lcm}(f,g)$ where $f\in\mathfrak{p}^a$ and $g\in\mathfrak{q}^b$.  Since $\mbox{lcm}(f,g)=\frac{fg}{\mbox{gcd}(f,g)}$ and since $a > b$, we know that $\mbox{lcm}(f,g)=fg'$ where $g'$ is a monomial in $\{z_1,\dots,z_r\}$.  Since the generators of $\mathfrak{p}^a$ all satisfy $\mathbf{h}_\mathfrak{p}\cdot\mathbf{u}=a$, we see that the generators of $I$ do as well.  Thus $\np(I)$ has codimension at least 1, and is contained in the hyperplane defined by $\mathbf{h}_\mathfrak{p}\cdot\mathbf{u}=a$.

    Let $\mathbf{H}$ be the vector with $-b$ in each coordinate corresponding to an $x_i$ and $a$ in each coordinate corresponding to a $z_j$.  We claim that the intersection of the halfspace defined by $\mathbf{H}\cdot\mathbf{u}\leq 0$ and $\np(I)$ is a facet of $\np(I)$.  We note that every generator of $I$ has the form $f'vg'$, where $f'$ is a monomial in $\{x_1,\dots,x_s\}$, $g'$ is a monomial in $\{z_1,\dots,z_r\}$, and $v$ is a monomial in $\{y_1,\dots,y_t\}$.  Furthermore, we note that $\deg f'=a-\deg v$ and $\deg g'=\max\{b-\deg v,0\}$.  Since $a>b$ we see then that the exponent vector of every monomial generator of $I$ satisfies $\mathbf{H}\cdot\mathbf{u}\leq 0$. Note that the generator $x_1^{a-b}y_1^b$ satisfies this inequality strictly, so $\mathbf{H}\cdot\mathbf{u}= 0$ is not a hyperplane of $\np(I)$.
    
 Furthermore, note that, for $i=1,\dots,t$, $j=2,\dots,r$, and $k=1,\dots,s$, the monomials $y_k^a$, $x_1^az_j^b$, and $x_i^az_1^b$ are all generators of $I$ whose exponent vectors satisfy $\mathbf{H}\cdot\mathbf{u}=0$.  The convex hull of these exponent vectors has codimension $2$.  Since $\np(I)$ has a codimension 2 face contained in the hyperplane $\mathbf{H}\cdot\mathbf{u}=0$, and a vertex not contained in that hyperplane, we see that $\np(I)$ has codimension 1.
 Thus we see that the halfspace $\mathbf{H}\cdot\mathbf{u}\leq0$ defines a facet of $\np(I)$.

 Thus by Theorem \ref{thm:spreadCalc}, it follows that $\ell(I)\le n-1$.
\end{proof}

Now, we aim to show equality.  To do so, we need a few more items:

\begin{defn} Let $I$ be an ideal in a Noetherian ring, $R$.  The \textit{arithmetic rank}, or \textit{arithmetical rank}, of an ideal $I$, denoted $\operatorname{ara} (I)$, is defined to be:
\[
    \operatorname{ara} (I) = \min\{g : \text{there exist }a_1, \ldots, a_g \in R \text{ such that } \sqrt{(a_1, \ldots, a_g)} = \sqrt{I} \},
\]
where $\sqrt{I}$ is the radical of $I$.
    
\end{defn}

\begin{rem} For an ideal $I$ in a Noetherian ring $R$, $\operatorname{ara} (I) \le \ell(I)$. This can be seen since, for a minimal reduction $J$ of $I$, $J$ is generated by $\ell(I)$ elements, and, since reductions have the same radical, $\sqrt J = \sqrt I$.  It is also clear that $\operatorname{ara} (I) = \operatorname{ara}(\sqrt{I})$.
    
\end{rem}

Lyubeznik \cite{Lyu} has the following result:
\begin{prop} If $I$ is a squarefree monomial ideal in the polynomial ring $R$, then
\[
    \operatorname{pd}(R/I) \le \operatorname{ara} I.
\]
    
\end{prop}

Combining the above with the Auslander-Buchsbaum formula, we have the following:

\begin{prop}\label{prop:inEqChain} Let $I$ be a monomial ideal in a polynomial ring of $n$ variables.  Then, $n - \operatorname{depth}(R/{\sqrt I}) \le \ell(I)$.
\end{prop}

With this, we can prove the following:

\begin{thm}\label{thm:twoPrimaryIdeals} Let $I = \mathfrak q_1 \cap \mathfrak q_2$ be a monomial primary decomposition of $I$.  Suppose the number of variables which make up generators of $\mathfrak q_1$ and $\mathfrak q_2$ is $m$.  Then $m - 1 \le \ell(I) \le m$.
    
\end{thm}

\begin{proof} By Theorem \ref{thm:embed}, we can assume that we are in a polynomial ring of $m$ variables.  This concludes the second inequality.

From Proposition \ref{prop:inEqChain}, it suffices to show that $\depth(R/\sqrt{I}) = \depth(R/\mathfrak p_1 \cap \mathfrak p_2) = 1$.

Let $\mathfrak p_1 = (x_1, \ldots, x_s, y_1, \ldots, y_t)$ and $\mathfrak p_2 = (y_1, \ldots, y_t, z_1, \ldots, z_r)$, where $r + s + t = m$.  Let $R' = R/\mathfrak p_1 \cap \mathfrak p_2$.  We can see that $\mathfrak p_1 \cap \mathfrak p_2 = (y_1, \ldots, y_t, x_iz_j : 1 \le i \le s, 1 \le j \le r)$.  Thus, $R' \cong \factor{K[x_1, \ldots, x_s, z_1, \ldots, z_r]}{(x_iz_j : 1 \le i \le s, 1 \le j \le r)}$.  Now, let $f = x_1 - z_1 + I\in R'$.  We will show that $S = (f)$ is a maximal regular sequence.

Suppose there is $0 \ne g \in R'$ so that $fg = 0$ in $R'$.  This corresponds to a $G \in R$ such that $(x_1-z_1)G \in I$, but $G \not\in I$.  Since $G \not \in I, G = P(x_1, \ldots, x_s) + Q(z_1, \ldots,z_r)$ for some polynomials $P$ and $Q$.  Then, we have that $(x_1 - z_1)(P(x_1, \ldots, x_s) + Q(z_1, \ldots,z_r)) = x_1P(x_1, \ldots, x_s) + x_1Q(z_1, \ldots,z_r)) - z_1P(x_1, \ldots, x_s) - z_1Q(z_1, \ldots,z_r)) \in I$.  The second and third terms in the sum are in $I$, so we must have $x_1P(x_1, \ldots, x_s) + z_1Q(z_1, \ldots, z_r) \in I$.  But no element of this form can exist, since none of the monomial generators show up in any of the terms.  Thus, no $G$ can exist, so $g = 0$ in $R'$.

Now, let $M = R'/(f) \cong \factor{K[x_1, \ldots, x_s, z_2, \ldots, z_r]}{(x_1^2, x_1x_2, \ldots, x_1x_s, x_iz_j : 1 \le i \le s, 2 \le j \le r)}$.  Here we can see that $\mathfrak m_{R} \in \Ass(M)$, as $x_1$ annihilates every variable in $\mathfrak m_{R}$.  Thus, $\depth(R') = 1$, and we are done.
\end{proof}

Immediately, we then get:

\begin{cor}\label{cor:twoPrimesEq} Let $I$ and $n$ be as in Theorem \ref{thm:twoPrimes}.  Then $\ell(I) = n - 1$.
\end{cor}

\begin{rem} Note that, using the notation in Corollary \ref{cor:twoPrimesEq}, if $t = 0$, then Theorem \ref{cor:twoPrimesEq} reduces to a simple case of Theorem \ref{thm:disGenPrimary}.  However, if $t > 0$, Theorem \ref{cor:twoPrimesEq} does not hold if the assumption were that $I$ is the intersection of two arbitrary primary ideals.  Indeed, if $R = K[x, y, z]$ and $I = (y^2z^2, y^4, xyz^2, x^3z^3) = (x^3, xy, y^2) \cap (y^4, y^2z^2, z^3)$, then $\ell(I) = 3$.
\end{rem}

\section{Open Questions}
\subsection{Monomial Reductions}  A remarkable fact about minimal reductions of monomial ideals is that they are often not monomial.  Consider the second ideal from Example \ref{ex:starter}, $I_2= (xy, x^4z^4, y^4z^4)$.  As discussed, $\ell(I_2) = 2$.  However, there is no monomial ideal $J$ that satisfies $\mu(J) = 2$.  Indeed, if $J$ did exist, it would be basic.  Thus, by Lemma \ref{lem:npMax}, $\np(J)$ is the unique compact facet of maximal dimension of $\NP(J)$.  But since $\NP(J) = \NP(I_2)$ and $\NP(I_2)$ has more than one maximal compact facet, $J$ is not basic, a contradiction.  (Note: According to Macaulay2, a minimal reduction of $I_2$ is $J = (51y^4z^4 + 2849xy, 34x^4z^4 - 885xy)$).

However, if we consider $I_3 =  (x^2y^2, x^4z^4, y^4z^4, xy^3z^2, x^3yz^2)$, then $J_3 = (x^2y^2, x^4y^4, y^4z^4)$ is a reduction of $I_3$.  Thus, for $I_3$, there \textit{do} exist monomial reductions with a smaller number of minimal generators.  In fact, $J_3$ is its minimal monomial reduction by the same argument above.

This leads to the following questions:
\begin{question}  Given a monomial ideal $I$,
\begin{enumerate}
    \item Does there exist a monomial ideal $J$ that is a minimal reduction of $I$?
    \item If $J$ is a reduction of $I$ that is monomial, how small can $\mu(J)$ be?  Are there multiple monomial reductions that attain this minimum?
\end{enumerate}
\end{question}

From a convex geometric point of view, this is equivalent to asking the minimal number of points so that the Minkowski sum of their convex hull and $\RR_{\ge 0}^n$ is $\NP(I)$.

\end{document}